\RequirePackage{etex}
\RequirePackage{easymat}

\documentclass[12pt]{elsarticle}
\usepackage{amsthm,amsmath,amssymb,mathrsfs}
\usepackage[all]{xy}

\usepackage[T1]{fontenc}

\usepackage[active]{srcltx}
\numberwithin{equation}{section}

\renewcommand{\le}{\leqslant}
\renewcommand{\ge}{\geqslant}

\newcommand{\mat}[1]{\begin{bmatrix}
   #1\end{bmatrix}}

\newcommand{\ff}{\mathbb F}
\newcommand{\cc}{\mathbb C}
\newcommand{\mc}{\mathcal}
\newcommand{\lin}{\: \frac{\quad\ }{}\:}
\newcommand{\lto}{\longrightarrow}

\newcommand{\ito}{\xrightarrow{\text{\raisebox{-3pt}{$\sim$}}}}

\newcommand{\un}{\underline}

 \newtheorem{thm}{Theorem}[section]
 \newtheorem{lem}[thm]{Lemma}
 \theoremstyle{definition}
 \newtheorem{defn}[thm]{Definition}

\begin{document}

\title{Lipschitz property for systems of linear mappings and bilinear forms}

\author[al]{Abdullah Alazemi} \address[al]{Department of Mathematics, Kuwait University, Safat 13060, Kuwait}
\ead{alazmi@sci.kuniv.edu.kw}

\author[al]{Milica An\dj eli\'{c}\corref{cor}} \ead{andjelic.milica@gmail.com}

\author[fo,fo1]{Carlos M. da Fonseca}
\address[fo]{Kuwait College of Science and Technology, Safat 13133, Kuwait}
\address[fo1]{University of Primorska, FAMNIT, Glagoljsa\v ska 8, 6000 Koper, Slovenia}
\ead{c.dafonseca@kcst.edu.kw, carlos.dafonseca@famnit.upr.si}

\author[ser]{Vladimir~V.~Sergeichuk}
\ead{sergeich@imath.kiev.ua}
\address[ser]{Institute of Mathematics, Tereshchenkivska 3,
Kiev, Ukraine}

\cortext[cor]{Corresponding author}

\begin{abstract}
Let $G$ be a graph with undirected and directed edges.
Its representation  is given by
assigning a vector space to each vertex, a bilinear form on the corresponding vector spaces to each directed edge, and a linear map to each directed edge.
Two representations $\mc A$ and $\mc A'$ of $G$ are called isomorphic if there is a system of linear bijections between the vector spaces corresponding to the same vertices that transforms $\mc A$ to $\mc A'$.  We prove that  if two representations are isomorphic and close to each other, then their isomorphism can be chosen close to the identity.
\end{abstract}

\begin{keyword}
15A21; 15A63; 16G20

\MSC
Systems of operators and forms; tensors of order two; Lipschitz property; representations of quivers and graphs
\end{keyword}

\maketitle

\section{Introduction}

If two square complex matrices are similar and close to each other, then they are similar via a matrix that can be chosen close to the identity matrix.
This fundamental fact is known as the \emph{local Lipschitz property for similarity} (Gohberg and Rodman \cite{g-r}); many mathematical constructions are based on it.

For example, the \emph{affine tangent space}  at the point $A\in\mathbb C^{n\times n}$ to the orbit of $A$ under similarity is the set $\{A+ XA-
AX\,|\,X\in{\mathbb C}^{n\times n}\}$ since by the Lipschitz property each matrix that is similar to $A$ and close to $A$ has the following form with a small $E$:
\begin{align*}
(I-E)^{-1}&A(I-E)=(I+E+E^2+\cdots)A(I-E)\\
&= A+(EA-AE)+E(EA-AE)+E^2(EA-AE)+\cdots
\\ &=
A+\underbrace{EA-AE}
_{\text{small}}
+\underbrace{E(I-E)^{-1}(EA-AE)}
_{\text{very small}}.
\end{align*}

Pierce and Rodman proved analogous local Lipschitz properties for congruence \cite{p-r3}, for simultaneous  congruence  and  simultaneous  unitary
congruence \cite{p-r2}, for matrix  group  actions  that  contain
simultaneous  similarity  and  simultaneous  equivalence \cite{p-r}, and joint similarity and congruence-like actions \cite{rod};  see also \cite{p-r1,rod1}.

The following theorem is a special case of \cite[Corollary 1.3(1)]{rod}.

\begin{thm}
\label{the1}
For each sequence $A_1,\dots,A_t$ of $n\times n$ complex matrices and each $r\in\{0,1,\dots,t\}$, there exist positive real numbers $\varepsilon$ and $K$ with the following property. Let $B_1,\dots,B_t$ be a sequence of $n\times n$ complex matrices such that $\sum_i\|B_i-A_i\|<\varepsilon $ and
\begin{equation}\label{afa}
\begin{split}
S^{-1}A_1S=B_1,\ \dots,\
S^{-1}A_rS=B_r,\\
S^TA_{r+1}S=B_{r+1},\ \dots,\
S^TA_tS=B_t
\end{split}
\end{equation}
for some nonsingular $S\in\cc^{n\times n}$. Then the equalities \eqref{afa} also hold
for some nonsingular $S$ satisfying $\|S-I\|\le K\sum_i\|B_i-A_i\|$.
\end{thm}

Here $\|\cdot\|$ is any matrix norm, although
Rodman \cite{rod} uses the
operator norm $\|A\|:= \max_{|x|=1} |Ax|$, in which $|\cdot|$ stands for the Euclidean norm of vectors.

The goal of this paper is to give an independent proof of a more general statement about representations of bidirected graphs (Theorem \ref{gjs}) using the Lipschitz property for matrix pairs with respect to similarity and two known facts about systems of linear mappings and bilinear forms.

\section{Representations of bidirected graphs}\label{ss2}

We consider systems  of linear mappings and bilinear forms  as representations of \emph{bidirected graphs}, which are the graphs with undirected, directed, and bidirected edges; for example,
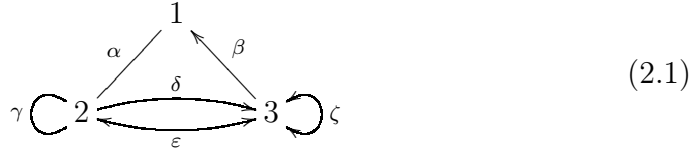
\begin{equation}\label{ksx}
\begin{split}
{\xymatrix@=2pc{
 &{1}&\\
 {2}\ar@(ul,dl)@{-}_{\gamma }
 \ar@{-}[ur]^{\alpha}
  \ar@/^/@{->}[rr]^{\delta }
 \ar@/_/@{<->}[rr]_{\varepsilon} &&{3}
 \ar@{->}[ul]_{\beta}
 \ar@(ur,dr)@{<->}^{\zeta}
 }}\end{split}
\end{equation}
The vertices of bidirected graphs that we consider are natural numbers $1,2,\dots,t$ ($t\ge 1$). Multiple edges and loops are allowed.

\begin{defn}[\cite{ser_izv}]
\label{qah}
Let $G$ be a bidirected graph.
\begin{itemize}
  \item
A \emph{representation} $\mc A$ of $G$ over a field $\ff$ is given by
assigning
\begin{itemize}
  \item a finite dimensional
vector space $\mc A_i$ over $\ff$ to each vertex $i$,

  \item
      a bilinear form $\mc A_{\alpha}:\mc A_j\times \mc A_i\to\ff$ to each undirected edge $\alpha:i\lin j$ with $i\le j$,

  \item  a linear mapping $\mc A_{\beta}:\mc A_i\to \mc A_j$ to each directed edge $\beta:i\lto j$,

  \item
      a bilinear form  on the dual spaces $\mc A_{\gamma}:\mc A_j^*\times \mc A_i^*$ to each bidirected edge $\gamma:i\longleftrightarrow j$ with $i\le j$ ($V^*$ denotes the dual space of all linear forms $V\to \ff$).

\end{itemize}

  \item
The \emph{dimension} of a representation $\mc A$ is the vector $\dim\mc A:=(\dim \mc A_1,\dots,\dim \mc A_t)$.

  \item
An \emph{isomorphism} $\varphi :\mc A\ito\mc B$ of two representations $\mc A$ and $\mc B$ of $G$ of the same dimension is a system $\varphi =(\varphi _1,\dots,\varphi _t)$ of linear bijections $\varphi_i:\mc A_i\to\mc B_i$ that transforms $\mc A$ to $\mc B$; that is,
$
\mc A_{\alpha }(x,y)=\mc B_{\alpha }(\varphi_j x,\varphi_i y)$ for each $\alpha :i\lin j$ ($i\le j$), $\varphi_j\mc A_{\beta}=\mc B_{\beta}\varphi_i$  for each  $\beta:i\lto j$, and
$\mc A_{\gamma}(\varphi_j^*x,\varphi_i^*y)
=\mc B_{\gamma}(x,y)$
for each  $\gamma:i\longleftrightarrow j$ ($i\le j)$.

  \item
A representation $\mc A$ of $G$ over a field $\ff$ in which all vector spaces $\mc A_i$ are of the form $\ff\oplus\dots\oplus \ff$ is called a
\emph{matrix representation}. (All forms and linear mappings of a matrix representation are given by matrices.)
\end{itemize}
\end{defn}

For example, a representation of \eqref{ksx}
is a system
\begin{equation}\label{ksev}
\begin{split}
{\raisebox{-20pt}{\text{$\mc A$}:}\quad
\xymatrix@=2pc{
 &{\mc A_1}&\\
  \save
!<-2mm,0cm>
 {\mc A_2}\ar@(ul,dl)@{-}_{\mc A_{\gamma} }
 \restore
 \ar@{-}[ur]^{\mc A_{\alpha}}
  \ar@/^/@{->}[rr]^{\mc A_{\delta }}
 \ar@/_/@{<->}[rr]_{\mc A_{\varepsilon}} &&{\mc A_3}
 \ar@{->}[ul]_{\mc A_{\beta}}
   \save
!<2mm,0cm>
 \ar@(ur,dr)@{<->}^{\mc A_{\zeta}}\restore
 }}\end{split}
\end{equation}
consisting of vector spaces $ \mc A_1,\, \mc A_2,\, \mc A_3$,  bilinear forms $\mc A_{\alpha }: \mc A_2\times \mc A_1\to \ff$ and $\mc A_{\gamma}: \mc A_2\times \mc A_2\to \ff$,
linear mappings $\mc A_{\beta}:\mc A_3\to \mc A_1$ and
 $\mc A_{\delta }:\mc A_2\to \mc A_3$,
 and
 bilinear forms on the dual spaces
 $\mc A_{\varepsilon}: \mc A_3^*\times \mc A_2^*\to \ff$ and $\mc A_{\zeta}: \mc A_3^*\times \mc A_3^*\to \ff$.

 The  forms and mappings in \eqref{ksev} can be considered as elements of tensor products $\mc A_{\alpha }\in \mc A_2^*\otimes \mc A_1^*$, $\mc A_{\gamma}\in \mc A_2^*\times \mc A_2^*$,
 $\mc A_{\beta}\in \mc A_3^*\otimes \mc A_1$,
 $\mc A_{\delta }\in\mc A_2^*\otimes\mc A_3$, and $\mc A_{\varepsilon}\in \mc A_3\otimes \mc A_2$, $\mc A_{\zeta}\in \mc A_3\otimes \mc A_3$; that is, as tensors of types $(0,2)$, $(1,1)$, and $(2,0)$.
Therefore, each representation of a bidirected graph is a system of vector spaces and tensors of order 2.

If all edges of $G$ are directed, then $G$ is a quiver and its representations are quiver representations. If all edges are directed and undirected, then $G$ is a mixed graph; its representations are considered in \cite{hor-ser_mixed}.

Each representation of a bidirected graph $G$ is isomorphic to a matrix representation; therefore, it is enough to study only matrix representations; they are given by matrices assigned to all edges.  The \emph{norm} $\|\mc A\|$ of a matrix representation $\mc A$ is the sum of norms of its matrices. For definiteness, we use the Frobenius norm
\begin{equation}\label{oim}
\|A\|:=\sqrt{\sum\nolimits_{i,j}|a_{ij}|^2}
\end{equation}
of a complex matrix $A=[a_{ij}]$,
although we could use an arbitrary matrix norm.

The main result of the paper is the following theorem, which is the \emph{local Lipschitz property for representations of bidirected graphs.}

\begin{thm}\label{gjs}
Let $\mc A$ be a complex matrix representation of dimension $(n_1,\dots,n_t)$ of a bidirected graph $G$.  There exist positive real numbers
$\varepsilon$ and $K$ with the following property: for every complex matrix representation $\mc B$ that is isomorphic to $\mc A$ and satisfies
$
\|\mc B -\mc A\| < \varepsilon$
there exists an isomorphism
$\varphi =(S_1,\dots,S_t):\mc B\ito\mc A$ such that
\[
\|S_1-I_{n_1}\|+\dots+\|S_t-I_{n_t}\| \le K\|\mc B - \mc A\|.
\]
\end{thm}

We show in Section \ref{mkm} that Theorem \ref{gjs}
follows easily from the Lipschitz property for complex matrix pairs with respect to similarity (which is proved in \cite{p-r}) and the following known facts:
\begin{itemize}
  \item the problem of classifying complex matrix pairs with respect to similarity transformations
\begin{equation*}\label{asq}
(M,N)\mapsto(R^{-1}MR,R^{-1}NR),\qquad R\text{ is nonsingular,}
\end{equation*}
       contains the problem of classifying complex representations of an arbitrary quiver (which is proved in \cite{bel-ser_compl,gel-pon}), and
  \item the problem of classifying complex representations of a bidirected graph $G$ is reduced to the problem of classifying complex representations of some quiver $\underline G$ (which is proved in \cite{ser_izv}).
\end{itemize}

\section{Proof of Theorem {\ref{gjs}}}\label{mkm}

We first prove the Lipschitz
property for quiver representations and then extend it to representations of
bidirected graphs.

\subsection{From matrix pairs to quiver representations}\label{ggg}
Let us prove the following theorem, which is the \emph{global Lipschitz property for representations of quivers.}
\begin{thm}\label{xgh}
Let $\mc A$ be a complex matrix representation of dimension $(n_1,\dots,n_t)$ of a quiver $Q$.  There exists a positive real number $K$ such that for every complex matrix representation $\mc B$ that is isomorphic to $\mc A$
there exists an isomorphism  $\varphi =(S_1,\dots,S_t):\mc B\ito\mc A$ satisfying
\[
\|S_1-I_{n_1}\|+\dots+\|S_t-I_{n_t}\| \le K\|\mc B - \mc A\|.
\]
\end{thm}

Let us consider
complex matrix representations
\begin{equation}\label{kshd}
\begin{split}
\mc A:\hspace{-30pt}
{\xymatrix@R=2pc@C=0pc{
 &{\cc^{n_1}}&\\
  \save
!<-3mm,0cm>
 {\cc^{n_2}}\ar@(ul,dl)@{->}_{C}
 \restore
 \ar@{->}[ur]^{A}
  \ar@/^/@{->}[rr]^{D}
 \ar@/_/@{->}[rr]_{E} &&{\cc^{n_3}}
 \ar@{->}[ul]_{B}
 \ar@{->}@(u,ur)^{F}
 \ar@{<-}@(dr,r)_{G}
 }}
\qquad\qquad
\mc B:\hspace{-30pt}
{\xymatrix@R=2pc@C=0pc{
 &{\cc^{n_1}}&\\
  \save
!<-3mm,0cm>
 {\cc^{n_2}}\ar@(ul,dl)@{->}_{C'}
 \restore
 \ar@{->}[ur]^{A'}
  \ar@/^/@{->}[rr]^{D'}
 \ar@/_/@{->}[rr]_{E'} &&{\cc^{n_3}}
 \ar@{->}[ul]_{B'}
 \ar@{->}@(u,ur)^{F'}
 \ar@{<-}@(dr,r)_{G'}
 }
}\end{split}
\end{equation}
(where $\cc^n:=\cc\oplus\dots\oplus\cc$ with $n$ summands)
of the
quiver
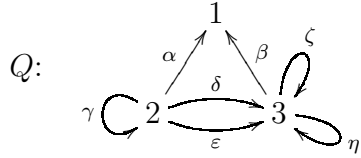
\begin{equation*}\label{kswd}
\begin{split}
{\raisebox{-20pt}{\text{$Q$}:}\quad}
{\xymatrix@R=2pc@C=1pc{
 &{1}&\\
 {2}\ar@(ul,dl)@{->}_{\gamma }
 \ar@{->}[ur]^{\alpha}
  \ar@/^/@{->}[rr]^{\delta }
 \ar@/_/@{->}[rr]_{\varepsilon} &&{3}
 \ar@{->}[ul]_{\beta}
 \ar@{->}@(u,ur)^{\zeta}
 \ar@{<-}@(dr,r)_{\eta}
 }}\end{split}
\end{equation*}

Define the matrix pair
\begin{equation}\label{nji}
(M,N(\mc A)):=
\left(
\mat{I_{n_1}\!&0&0&0\\0&\!\!2I_{n_2}\!\!&0&0\\
0&0&\!\!3I_{n_3}\!\!&0\\0&0&0&\!\!4I_{n_3}},
\mat{A&B&0&0\\C&0&0&0\\
0&D&F&I\\0&E&G&0}
\right)
\end{equation}
by the representation $\mc A$.

\begin{lem}[{\cite[Theorem 2.1]{bel-ser_compl}}]
\label{njf}
The matrix representations $\mc A$ and $\mc B$ in \eqref{kshd} are isomorphic if and only if the  pairs $(M,N(\mc A))$ and $(M,N(\mc B))$ are similar.
\end{lem}

\begin{proof}
Each matrix representation that is isomorphic to $\mc A$ has the form
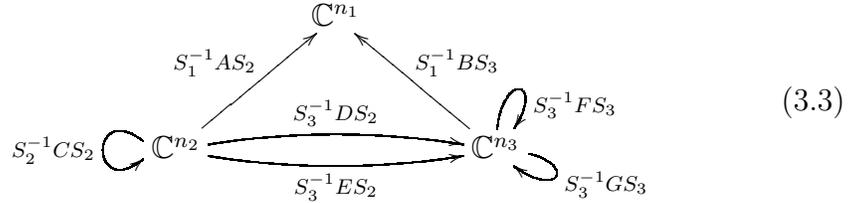
\begin{equation}\label{ksd2}
\begin{split}
{\xymatrix@=3pc{
 &{\cc^{n_1}}&\\
   \save
!<-3mm,0cm>
 {\cc^{n_2}}\ar@(ul,dl)
 @{->}_{S_2^{-1}CS_2}
 \restore
 \ar@{->}[ur]^{ S_1^{-1}AS_2}
  \ar@/^/@{->}[rr]^{S_3^{-1}DS_2}
 \ar@/_/@{->}[rr]_{ S_3^{-1}ES_2} &&{\cc^{n_3}}
 \ar@{->}[ul]_{ S_1^{-1}BS_3}
 \ar@{->}@(u,ur)^(.7){S_3^{-1}FS_3}
 \ar@{<-}@(dr,r)_{S_3^{-1}GS_3}
 }}\end{split}
\end{equation}
where $S_1,S_2,S_3$ are nonsingular matrices, which can be considered as the change of basis matrices.

$\Longleftarrow$.
Let
\begin{equation}\label{mkg}
R^{-1}MR=M,\qquad R^{-1}N(\mc A)R
=N(\mc B)
\end{equation}
for some nonsingular $R$. The equality $R^{-1}MR=M$ implies that $R=S_1\oplus S_2\oplus S_3\oplus S_4$. We conclude from  the second equality in \eqref{mkg} that $R$ acts on $N(\mc A)$ by similarity transformations as follows:
\[
\begin{MAT}(b){ccccc}
&\scriptstyle S_1&
\scriptstyle S_2&
\scriptstyle S_3&\scriptstyle S_4
                                 \\
\scriptstyle  S_1^{-1}&0&A&B&0
\\
\scriptstyle  S_2^{-1}
&0&C&0&0
\\
\scriptstyle S_3^{-1}
&0&D&F&I
\\
\scriptstyle S_4^{-1}
&0&E&G&0
\addpath{(1,0,4)uuuu%
rrrrdddd%
llll}
\addpath{(1,2,2)rrrr}
\addpath{(1,1,2)rrrr}
\addpath{(1,3,2)rrrr}
\addpath{(2,0,2)uuuu}
\addpath{(3,0,2)uuuu}
\addpath{(4,0,2)uuuu}
\\
\end{MAT}
\quad
\raisebox{-6pt}{\text{$\mapsto$}}
\quad
\begin{MAT}(b){cccc}
&
&&                                 \\
\rule{0pt}{9pt}0&A'&B'&0
\\
0&C'&0&0
\\
0&D'&F'&I
\\
0&E'&G'&0
\addpath{(0,0,4)uuuu%
rrrrdddd%
llll}
\addpath{(0,2,2)rrrr}
\addpath{(0,1,2)rrrr}
\addpath{(0,3,2)rrrr}
\addpath{(1,0,2)uuuu}
\addpath{(2,0,2)uuuu}
\addpath{(3,0,2)uuuu}
\\
\end{MAT}
\]
Since these transformations preserve the $(3,4)$th block $I$, $S_3=S_4$. We have that $\mc
B$ is of the form \eqref{ksd2}, and so $\mc B$ is isomorphic to $\mc A$.

$\Longrightarrow$.
Let $\mc B$ be isomorphic to $\mc A$. Then $\mc B$ is of the form
\eqref{ksd2}, and so $R^{-1}(M,N(\mc A))R=(M,N(\mc B))$ with $R:=S_1\oplus S_2\oplus S_3\oplus S_3$.
\end{proof}

\begin{proof}[Proof of Theorem \ref{xgh}] For the sake of clarity, we prove Theorem \ref{xgh} for quiver representations \eqref{kshd}; its proof for representations of an arbitrary quiver is analogous.

By \cite[Theorem 3.1]{p-r},
the global Lipschitz property holds for complex matrix pairs with respect to similarity. Applying it to the matrix pair \eqref{nji}, we find that there is a positive real number $K$ with the following property:  for every complex matrix pair $(M',N')$ that is similar to $(M,N(\mc A))$
there exists a nonsingular matrix $R$ satisfying
\begin{gather}\label{cde}
R^{-1}MR=M',\qquad R^{-1}N(\mc A)R
=N',\\ \label{cde1}
\|R-I\|\le K(\|M' - M\|+\|N' - N(\mc A)\|).
\end{gather}

Let a representation $\mc B$ of the form \eqref{kshd} be isomorphic to $\mc A$. By Lemma \ref{njf}, the pairs $(M,N(\mc A))$ and $(M,N(\mc B))$ are similar, which allows us to take $(M',N'):=(M,N(\mc B))$ in
\eqref{cde} and \eqref{cde1}. Since $R^{-1}MR=M$, $R:=S_1\oplus S_2\oplus S_3\oplus S_4$. The equality $R^{-1}N(\mc A)R
=N(\mc B)$ ensures that $S_3=S_4$. Hence
$\mc B$ is the representation \eqref{ksd2} and $\varphi:=(S_1,S_2,S_3):\mc B\ito\mc A$.
By \eqref{oim} and \eqref{cde1},
\begin{align*}
\|\varphi-1\|&=\|S_1-I\| +\|S_2-I\|+\|S_3-I\|\\&\le \|R-I\|
\le K\|N(\mc B) - N(\mc A)\|\\&\le K(\|A'-A\|+\|B'-B\|+\dots+\|G'-G\|)
\\&=K\|\mc B-\mc A\|.\tag*{\qedhere}
\end{align*}
\end{proof}

\subsection{From quiver representations to representations of bidirected graphs}\label{ggg1}

Let us recall the method that reduces the problem of classifying systems of linear mappings and forms to the problem of classifying systems of linear mappings. This method was developed by Roiter and Sergeichuk \cite{roi,ser_izv}; it is used in \cite{hor-ser_bilin,ser_isom}.

For every bidirected graph
${G}$, we denote by
$\underline{G}$ the
quiver obtained from
$G$ by replacing
\begin{itemize}\parskip=-3pt
  \item
each vertex $i$ of
$G$ by the vertices
$i$ and $i^*$,
  \item
each arrow
$\alpha:  i\lto j$
by the arrows
$\alpha: i\lto j$
and $\alpha^*:
j^*\lto i^*$,
  \item
each edge $\beta:
i\lin\, j\ (i\le j)$
by $\beta: j\lto
i^*$ and
$\beta^*: i\lto
j^*$,
  \item
each edge
$\gamma:
i\longleftrightarrow
j$ ($i\le j$)  by
$\gamma: j^*\lto
i$ and $\gamma^*:
i^*\lto j$.
\end{itemize}
We put $i^{**}:=i$ and $\alpha^{**}:=\alpha$ for all vertices and arrows. The mappings $i\mapsto i^*$ and $\alpha\mapsto\alpha^*$ are  involutions on the sets of vertices and arrows of the quiver $\underline{G}$.

For example,
\begin{equation}\label{4.1}
\begin{split}
{\xymatrix@R=4pt{
 &{1}\ar[dd]_{\alpha}
 \ar@{-}@/^/[dd]^{\beta}\\
{G}:\!\!\!\! &
  \\
 &{2}
\save !<1mm,0cm>
\ar@(ur,dr)@{<->}^{\gamma}
\restore}}
\qquad\qquad
{\xymatrix@R=4pt{
 &{1}\ar[dd]_{{\alpha}}
 \ar[ddrr]^(.25){{\beta}}&
 &{1^*} \\
 {\underline{G}:}\!\!&&\\
 &{2}\ar[uurr]^(.75){{\beta}^*}
 &&{2^*}\ar[uu]_{{\alpha}^*}
\ar@<-0.4ex>[ll]_{{\gamma}}
 \ar@<0.4ex>[ll]^{{\gamma}^*}
 }}
\end{split}
\end{equation}

For every complex matrix representation $\mathcal
A$ of a bidirected graph $G$, we define
the complex matrix representation
$\underline {\mathcal A}$ of
$\underline {G}$ with the same $A_{\alpha}$ and with $A_{\alpha^*}:=A_{\alpha}^T$ for each edge $\alpha $ of $G$.
For example,
\begin{equation}\label{4.df}
\begin{split}
{\xymatrix@R=4pt{
 &{1}\ar[dd]_{A}
 \ar@{-}@/^/[dd]^{B}\\
{\mathcal A}:\!\!&
  \\
 &{2}
\save !<0.9mm,0cm>
\ar@(ur,dr)@{<->}^{C}
\restore}}
   \qquad   \qquad
{\xymatrix@R=4pt{
 &{1}\ar[dd]_{A}
 \ar[ddrr]^(.25){B}&
 &{1^*} \\
 {\underline{\mathcal A}:}\!\!&&\\
 &{2}\ar[uurr]^(.75){B^T}
 &&{2^*}\ar[uu]_{A^T}
 \ar@<0.4ex>[ll]^{C^T}
 \ar@<-0.4ex>[ll]_{C}
 }}
 \end{split}
\end{equation}
for complex
matrix representations  of \eqref{4.1}.

The following lemma is a special case of \cite[Theorem]{roi} or \cite[Theorem 2]{ser_izv}.

\begin{lem}\label{jyr}
Let $\mc A$ and $\mc B$ be two complex matrix representations of a bidirected graph $G$.  Then $\mc A$ and $\mc B$ are isomorphic if and only if $\underline{\mc A}$ and $\underline{\mc B}$ are isomorphic.
\end{lem}

\begin{proof}  For clarity, we prove Lemma \ref{jyr} for matrix representations of the bidirected graph $G$ given in \eqref{4.1}; its proof for matrix representations of an arbitrary bidirected graph is analogous.

Let $\mc B$ be obtained from $\mc A$ in \eqref{4.df} by replacing $A,B,C$ with $A',B',C'$.

$\Longrightarrow$.
Let $\varphi=(\Phi_1,\Phi_2):\mc B\ito\mc A$. Then
\begin{equation}\label{hhm}
\begin{split}
    \xymatrix@R=15pt@C=40pt{
{1}
       \ar@/^1.5pc/@{-->}%
       [rrr]^{\Phi_1}
\ar[d]_{%
\smash{\textstyle\underline{\varphi}:}\quad
 A'}\ar[rd]^(.25){B'}
&{1^*}
       \ar@/^1.5pc/@{-->}%
       [rrr]^{\Phi_{1}^{-T}}
&&
{1}
\ar[d]_{A}\ar[rd]^(0.25){B}
&{1^*}
                               \\
{2}
             \ar@/_1.5pc/@{-->}%
             [rrr]_{\Phi_2}
\ar[ru]^(.75){B^{\prime T}}&
{2^*}
 \ar@<0.4ex>[l]^(0.25){C^{\prime T}}
 \ar@<-0.4ex>[l]_{C'}
             \ar@/_1.5pc/@{-->}%
             [rrr]_{\Phi_{2}^{-T}}
\ar[u]_{A^{\prime T}}&&
{2}\ar[ru]^(0.75){B^T}&
{2^*}\ar[u]_{A^T}
 \ar@<0.4ex>[l]^(0.65){C^{T}}
 \ar@<-0.4ex>[l]_{C}
}
\end{split}
\end{equation}
is the isomorphism
$
\underline{\varphi}=(\Phi_1,\Phi_2,
\Phi_{1^*},\Phi_{2^*}):=
(\Phi_1,\Phi_2,
\Phi_{1}^{-T},\Phi_{2}^{-T})
:\underline{\mc B}\ito \underline{\mc A}.
$

$\Longleftarrow$.
Let
\begin{equation}\label{kib}
\psi=(\Psi_1,\Psi_2,\Psi_{1^*}, \Psi_{2^*})=(P,Q,R,S):\un{\mc B}\ito\un{\mc A}\,,
\end{equation}
Then $\psi^{\circ}:=(R^T,S^T,P^T,Q^T):\un{\mc A}\ito\un{\mc B}$ and \[\psi^{\circ}\psi= (R^TP,S^TQ,P^TR,Q^TS):\un{\mc B}\ito\un{\mc B}.\] Take a nonzero polynomial $f\in\cc[x]$ and consider the morphisms of quiver representations
\begin{align*}
f(\psi^{\circ}\psi)&:=
(f(R^TP),f(S^TQ),f(P^TR),f(Q^TS)):\un{\mc B}\to\un{\mc B},\\
\psi^{-\circ}f(\psi^{\circ}\psi)&:=
(R^{-T}f(R^TP),S^{-T}f(S^TQ),\\&\hspace{2cm} P^{-T}f(P^TR),Q^{-T}f(Q^TS)):\un{\mc B}\to\un{\mc A}.
\end{align*}
We must choose the polynomial $f$ such that $\psi^{-\circ}f(\psi^{\circ}\psi)$ has the form \eqref{hhm}; that is, \begin{equation}\label{ffc}
\begin{split}
(R^{-T}f(R^TP))^{-T}=P^{-T}f(P^TR),
 \\
 (S^{-T}f(S^TQ))^{-T}=Q^{-T}f(Q^TS).
\end{split}
 \end{equation}
The first equality is equivalent to $ Rf(P^TR)^{-1}=P^{-T}f(P^TR)$, and so the equalities \eqref{ffc} are equivalent to $P^TR=f(P^TR)^2,\
 Q^TS=f(Q^TS)^2$, which are equivalent to
\begin{equation*}\label{fmc}
 P^TR\oplus Q^TS=f(P^TR\oplus Q^TS)^2.
\end{equation*}

Such an $f$ exists since for each nonsingular complex matrix $M$ there is a polynomial in $M$ whose square is $M$; see Kaplansky \cite[Theorem 68]{kapl}. We obtain $\psi^{-\circ}f(\psi^{\circ}\psi)$ of the form \eqref{hhm}; its first two matrices define the isomorphism
\begin{equation}\label{jhe}
(R^{-T}f(R^TP),S^{-T}f(S^TQ)): {\mc B}\ito {\mc A},
\end{equation}
which proves Lemma \ref{jyr}.
\end{proof}

\begin{proof}[Proof of Theorem \ref{gjs}] For clarity, we prove Theorem \ref{gjs} for matrix representations of the bidirected graph $G$ in \eqref{4.1}.

Let $\mc A$ be a complex matrix representation of $G$.
We must prove that there exist positive numbers
$\varepsilon$ and $K$ with the following property: for every matrix representation $\mc B$ that is isomorphic to $\mc A$ and satisfies
$
\|\mc B -\mc A\| < \varepsilon$
there exists an isomorphism
$\varphi =(\Phi _1,\Phi _2):\mc B\ito\mc A$ such that
\[
\|\Phi_1-I\|+\|\Phi_2-I\| \le K\|\mc B - \mc A\|.
\]

Let $\mc A'$ be any matrix representation of $G$ that is isomorphic to $\mc A$.
Then $\un{\mc A}$ and $\un{\mc A'}$ are isomorphic representations of $\underline G$ and
\begin{equation*}\label{vcm}
\|\un{\mc A}'-\un{\mc A}\|=2\|{\mc A}'-{\mc A}\|.
\end{equation*}
By Theorem \ref{xgh},
there exists $K>0$ (the same for all $\mc A'$) and an isomorphism  \eqref{kib} satisfying
\[
\delta :=\|P-I\|+
\|Q-I\|+
\|R-I\|+
\|S-I\|
\le K\|\un{\mc A}'-\un{\mc A}\|.
\]
Let us prove that \eqref{jhe} is a desired isomorphism; that is, there exists $M>0$ (the same for all $\mc A'$ that are sufficiently close to $\mc A$) such that
\begin{equation}\label{nhi}
\|R^{-T}f(R^TP)-I\|\le M\delta ,\quad
\|S^{-T}f(S^TQ)-I\|\le M\delta.
\end{equation}

Let us find $M$ for the first inequality:

(i)
Write $R=I+\Delta$ and $R^{-1}=I+\nabla$, in which
$\Delta$ and $\nabla$ are sufficiently small matrices.
Then
\begin{gather*}
I=RR^{-1}=(I+\Delta)(I+\nabla)
=I+\Delta+\nabla+\Delta\nabla,\\ \Delta+\nabla+\Delta\nabla=0,\ \ \text{and}\ \
\|\nabla\|=\|\Delta+\Delta\nabla\|
\le \|\Delta\|+\|\Delta\nabla\|.
\end{gather*}
Since $\|\Delta\nabla\|\le \|\Delta\|\|\nabla\|\le \|\Delta\|$ for a sufficiently small $\nabla$, we have $\|\nabla\|\le 2\|\Delta\|$.
Hence
\begin{equation*}\label{rru}
\|R^{-T}-I\|=
\|R^{-1}-I\|=\|\nabla\|\le 2\|\Delta\|\le 2\|R-I\|\le 2\delta .
\end{equation*}

(ii)
For each $X,Y\in\cc^{n\times n}$, we have
\begin{equation}\label{kjs}
XY-I=(X-I)(Y-I)+(X-I)+(Y-I).
\end{equation}
Taking $X=R^T$ and $Y=P$, we get
$\|R^TP-I\|\le \delta^2+2\delta\le 3\delta$ for a sufficiently small $\delta$.

(iii)
Write $f(R^TP)=I+D$, in which $D$ is a sufficiently small matrix. Since $f(R^TP)^2=R^TP$, we have
$R^TP=(I+D)^2=I+2D +D^2$, and so
\begin{align*}
 \|R^TP-I\|&=\|2D +D^2\|\ge 2\|D\| -\|D^2\|\\&\ge 2\|D\| -\|D\|^2 \ge \|D\|+(\|D\| -\|D\|^2)\ge \|D\|.
\end{align*}
Hence, $\|f(R^TP)-I\| =\|D\|\le \|R^TP-I\|\le 3\delta$.

By \eqref{kjs}, (i), and (iii),
\begin{align*}
&\|R^{-T}f(R^TP)-I\|\le \|R^{-T}-I\|\|f(R^TP)-I\|
+\|R^{-T}-I|
\\&\qquad+\|f(R^TP)-I\|\le 2\delta 3\delta +2\delta +3\delta
\le 6\delta +2\delta +3\delta=11\delta.
\end{align*}
Therefore, we can
take $M=11$ in the first inequality of \eqref{nhi}. Analogously, we can take  $M=11$ in the second inequality.
\end{proof}

\end{document}